\documentclass[11pt]{article}
\usepackage[utf8]{inputenc}
\usepackage{graphicx, color}
\usepackage{mymacro}
\usepackage{wasysym}
\usepackage[all]{xy}
\usepackage{typearea}
\renewcommand{\Sh}{\mathrm{Sh}}

\renewcommand{\Mod}{\mathrm{Mod}}

\newcommand{\Perf}{\mathrm{Perf}}

\newcommand{\Spc}{\mathrm{Spc}}
\newcommand{\Mon}{\mathrm{Mon}}

\title{A remark on monoidal structure \\and homological mirror symmetry}
\author{Tatsuki Kuwagaki}
\date{\today}

\begin{document}

\maketitle
\begin{abstract}
For a symplectic geometry $X$, suppose the (derived) Fukaya category $\mathrm{Fuk}(X)$ of $X$ is equipped with a monoidal structure. Then its Balmer spectrum recovers a mirror $Y$ of $X$ if there exists homological mirror symmetry $\mathrm{Fuk}(X)\cong D^b\mathrm{coh}(Y)$ and the monoidal structure is the mirror of the standard one of $D^b\mathrm{coh}(Y)$. In this short note, we fill one gap of this story in the literature: we show that the monoidal structure determines the homological mirror functor $\mathrm{Fuk}(X)\to D^b\mathrm{coh}(Y)$.
\end{abstract}

\section{Introduction}
\subsection{Monoidal structure and homological mirror symmetry}
Let $X$ be a symplectic geometry with some data to define its (derived) Fukaya category $\Fuk(X)$ over a fixed base field $\bK$. Suppose $X$ satisfies ``homological mirror symmetry"~\cite{KontsevichHMS} i.e., there exists a smooth algebraic variety $Y$ over $\bK$ such that its derived category of coherent sheaves $D^b\coh(Y)$ is equivalent to $\Fuk(X)$. Since $D^b\coh(Y)$ is equipped with a monoidal structure induced by tensor over the structure sheaf, the Fukaya category is given a monoidal structure induced by that of $D^b\coh(Y)$. Note that there can be non-isomorphic $Y$'s whose $D^b\coh$ are equivalent to $\Fuk(X)$. The monoidal structures on $\Fuk(X)$ induced by non-isomorphic $Y$'s are indeed non-isomorphic by Balmer's reconstruction theorem~\cite{Balmer}.

Hence $\Fuk(X)$ may have several monoidal structures. In other words, there is no unique canonical monoidal structure in general. What is the geometric meaning of this non-canonicity? An expected picture is the following: In the situation of mirror symmetry, $X$ admits a Lagrangian torus fibration (SYZ fibration) $\pi\colon X\to B$ whose dual fibration $\pi^\vee\colon Y\to B$ gives a mirror geometry $Y$~\cite{SYZ}. So, different fibrations give different mirrors. Moreover, it is expected that the fiberwise addition on the torus fibration induces a monoidal structure on $\Fuk(X)$ which is the mirror of the standard monoidal structure on $D^b\coh(Y)$. Some progress in this direction are made e.g., in \cite{Subotic, AbouzaidBottmanNiu}.

The SYZ fibration $\pi\colon X\to B$ is also expected to give a homological mirror symmetry functor~\cite{Fukayafamilty, Abouzaidfamilyfloer}: the family Floer mirror functor. Roughly speaking, a point $y$ in $Y$ over $b\in B$ corresponds to an object $L_y$ of $\Fuk(X)$ supported on $\pi^{-1}(b)$. Then, for an object $L\in \Fuk(X)$, the stalk of mirror coherent sheaf $\cE$ at $y$ is given by $\Hom_{\Fuk(X)}(L_y, L)$.

Now we have the following diagram of implications:
\begin{equation}\label{diagram}
    \xymatrix{
    &\text{SYZ fibration}\ar[rd]^{\text{Family Floer functor}
    }\ar[ld]_{\text{Fiberwise addition}}&\\
    \text{Monoidal structure on $\Fuk(X)$}&&\ar[ll]^{\text{pulling-back the monoidal structure of $D^b\coh(Y)$}} \substack{\text{Homological mirror symmetry functor}\\  \Fuk(X)\to D^b\coh(Y)}
    }
\end{equation}
The diagram is expected to be commutative. In the case of toric mirror symmetry, it is checked in the context of the coherent-constructible correspondence~\cite{FLTZ1, FLTZ2, KuwagakiCCC}.\footnote{Although it is not given by the family Floer functor, the relevant functor has expected properties.}

In this short note, our main claim is the following:
\begin{claim}
    The horizontal arrow in (\ref{diagram}) is invertible.
\end{claim}
We will explain the precise meaning of the claim in the rest of the introduction.

\subsection{Main result}
Our main result is a slight refinement of Balmer's construction~\cite{Balmer} in the enhanced setting. Let $\cT$ be a $\bK$-linear tensor triangulated category with an enhancement. From the homotopy category of $\cT$, one can produce a ringed space $(\Spc_{\otimes}(\cT), \cO_{\Spc_{\otimes}(\cT)})$, called \emph{Balmer spectrum}. 
If $\cT$ is the derived category of perfect complexes over some Noetherian scheme $Y$, with the standard monoidal structure, its Balmer spectrum recovers $Y$.

We denote the derived category of $\cO_{\Spc_{\otimes}(\cT)}$-modules over $\Spc_{\otimes}(\cT)$ by $D(\cO_{\Spc_{\otimes}(\cT)})$.

\begin{theorem}\label{theorem:Balmer_functor}
Suppose $\Spc_\otimes(\cT)$ is hypercomplete and the higher structure sheaf is classical. 
Then there exists a canonical functor $m_{\cT, \otimes} \colon \cT\to D(\cO_{\Spc_{\otimes}(\cT)})$.

If $\cT$ is the derived category of perfect complexes over some Noetherian scheme, with the standard monoidal structure, $m_{\cT,\otimes}$ is the standard inclusion functor.

\end{theorem}
See the body of the paper for the precise meaning of the assumptions.

\subsection{A Consequence on homological mirror symmetry}
Let $X$ be a symplectic geometry with some additional data to define the $\bZ$-graded $\bK$-linear Fukaya category $\Fuk(X)$.

\begin{corollary}
Suppose $F\colon \Fuk(X)\xrightarrow{\cong}D^b\coh(Y)$ for a smooth variety $Y$. The monoidal structure $\otimes_F$ on $\Fuk(X)$ induced from the canonical one on $D^b\coh(Y)$ recovers $F$, i.e., $m_{Fuk(X),\otimes_F}=F$.
\end{corollary}
\begin{proof}
    We have the following commutative diagram
    \begin{equation}
        \xymatrix{
        \Fuk(X)\ar[r]^F_{\cong}\ar[dr]_{m_{\Fuk(X),\otimes_F}} & D^b\coh(Y)\ar[d]^{m_{D^b\coh(Y), \otimes}}\\
        &D(\cO_{Y}).
        }
    \end{equation}
    By Theorem~\ref{theorem:Balmer_functor}, the right vertical arrow is the standard inclusion. This completes the proof.
\end{proof}

\subsection*{Acknowledgment}
This work was inspired by some discussion with some people including Daigo Ito,  Hayato Morimura, Genki Ouchi, Atsushi Takahashi, and Hiro Lee Tanaka. I thank all of them. The author would like to give special thanks to Genki Ouchi for many valuable comments and stimulating discussion on the draft.
This work is supported by JSPS KAKENHI Grant Numbers 22K13912, 23K25765, and 	25K21660.

\section{Construction}

\subsection{Balmer spectrum}
Let $\bK$ be a field.
Let $\cT$ be a tensor triangulated category over $\bK$. We denote the tensor (=monoidal operation) of $\cT$ by $\otimes$, and the tensor unit by $1_\cT$. Balmer~\cite{Balmer} constructed a ringed space from $\cT$, called the \emph{Balmer spectrum}. The definition is as follows.

A triangulated subcategory of $\cT$ is \emph{thick} if it is closed under taking direct summands. A thick subcategory $\cS$ is \emph{prime thick ideal} if the following conditions are satisfied,
\begin{enumerate}
    \item $\cE\in \cS, \cF\in \cT$ implies that $\cE\otimes \cF\in \cS$.
    \item  $\cE\otimes \cF\in \cS$ implies $\cE\in \cS$ or $\cF\in \cS$.
\end{enumerate}
Let $\Spc_{\otimes}(\cT)$ be the set of prime thick ideals of $\cT$.

Let $\bE$ be a set of objects in $\cT$. We set
\begin{equation}\label{eq:Balmer_basis}
    Z(\bE):=\lc \cS\in \Spc_{\otimes}(\cT)\relmid \cS\cap \bE= \varnothing\rc.
\end{equation}
It is easy to see $Z(\varnothing)=\Spc_{\otimes}(\cT), Z(\mathrm{Ob}(\cT))=\varnothing, \bigcap_iZ(\bE_i)=Z(\bigcup_i\bE_i), Z(\bE_1)\cap Z(\bE_2)=Z(\lc\cE_1\oplus \cE_2\relmid \cE_1\in \bE_1, \cE_2\in \bE_2 \rc)$. Hence $Z$ defines a topology on $\Spc_{\otimes}(\cT)$, called \emph{Balmer topology}.

For each open subset $U\subset \Spc_\otimes(\cT)$, we assign $\End_{\cT_U}(1_{U})$ where $1_{U}$ is the image of the monoidal unit $1_\cT$ by the quotient functor $[-]_U\colon \cT\to \cT_U:=\cT/\bigcap_{\cP\in U}\cP$. This forms a presheaf on $\Spc_\otimes(\cT)$. The sheafification is denoted by $\cO_{\cT, \otimes}$ and is called the structure sheaf on $\Spc_{\otimes}(\cT)$. By a general nonsense, this forms a sheaf of commutative $\bK$-algebras on $\Spc_{\otimes}(\cT)$.

\begin{definition}[Balmer spectrum] The ringed space $(\Spc_\otimes(\cT), \cO_{\cT,\otimes})$ is the \emph{Balmer spectrum} of $\cT$.
\end{definition}

\subsection{The canonical functor}
In the following, we work in an enhanced setting. Note that infinity-categorical instances of Balmer spectrum are also discussed in e.g., \cite{matsukawa2025spectrumstableinfinitycategories, aoki2025higherzariskigeometry}.

Let $\cC$ be a symmetric monoidal stable $\infty$-category over $\bK$. Then the homotopy category $\pi_0(\cC)$ is a tensor triangulated category over $\bK$, and one can consider its Balmer spectrum $(\Spc_\otimes(\pi_0(\cC)), \cO_{\pi_0(\cC), \otimes})$. For this ringed space, one can form the derived category of $\cO_{\pi_0(\cC), \otimes}$-modules $D(\cO_{\pi_0(\cC),\otimes})$.

Let $\mathrm{Op}_{\Spc_{\otimes}(\pi_0(\cC))}$ be the poset (will be viewed as an $\infty$-category) of the open subsets of $\Spc_{\otimes}(\pi_0(\cC))$. Then the assignment $U\mapsto \cC_U$ forms a functor 
\begin{equation}
    \cC_\bullet\colon \mathrm{Op}_{\Spc_{\otimes}(\pi_0(\cC))}^{op}\to (Cat^{st}_{\bK})_{\cC /}.
\end{equation}
Here $(Cat^{st}_{\bK})_{\cC /}$ is the category of stable categories under $\cC$.
Note that the sheafification of this functor (without $\cC/$) appears in \cite{aoki2025higherzariskigeometry} and is called the structure sheaf in higher Zariski geometry.

Since $\cC$ is pointed by the unit, we have a functor
\begin{equation}
    (Cat^{st}_{\bK})_{\cC /}\to Fun(\cC,\Mod(\bK)); \cC'\mapsto \Hom_{\cC'}(1_{\cC'}, -)
\end{equation}
where $Fun(\cC,\Mod(\bK))$ is the functor category from $\cC$ to the derived $\bK$-modules $\Mod(\bK)$ (i.e., the category of $\cC^{op}$-modules over $\bK$), $1_{\cC'}$ is the image of the unit $1_\cC$ under the structure map $\cC\to \cC'$. Composing these two functors, we obtain a functor from  $\mathrm{Op}_{\Spc_{\otimes}(\pi_0(\cC))}^{op}$ to $Fun(\cC,\Mod(\bK))$. In other words, we obtain
\begin{equation}
    m_{\cC}^{pre}\colon \cC\to \mathrm{Psh}(\Spc_{\otimes}(\pi_0(\cC)), \Mod(\bK)).
\end{equation}
An object $U\mapsto \End(1_{\cC_U})$ defines an algebra object $\cO_\cC^{pre}$ in $\mathrm{Psh}(\Spc_{\otimes}(\pi_0(\cC)), \Mod(\bK))$. By the definition, $m_{\cC}^{pre}$ lands in the modules over $\cO_\cC^{pre}$ in $\mathrm{Psh}(\Spc_{\otimes}(\pi_0(\cC)), \Mod(\bK))$.

By applying the sheafification, we obtain $ m_{\cC}\colon \cC\to \Sh(\Spc_{\otimes}(\pi_0(\cC)), \cO_C)$ where $\Sh(\Spc_{\otimes}(\pi_0(\cC)), \cO_C)$ is the category of modules over the sheafification $\cO_\cC$ of $\cO_\cC^{pre}$ in $\Sh(\Spc_{\otimes}(\pi_0(\cC)), \Mod(\bK))$. We call $\cO_{\cC}$, the \emph{higher structure sheaf} of $\cC$.

\begin{definition}\label{def:canonical}
    We call $m_\cC$, the \emph{canonical functor}.
\end{definition}

\begin{remark}
    A similar construction at the level of homotopy category is due to Rowe~\cite{Rowe}, and is called the associated sheaf functor. The canonical functor here can be considered as a higher version of the associated sheaf functor.
\end{remark}

\begin{remark}
Although we do not use in the rest of the paper, let us state some formal properties of $m_\cC$.
\begin{enumerate}
\item From the definition, it is clear that $m_\cC$ is lax monoidal.
    \item     One can check that the canonical functor is fully faithful if $\cC$ is locally monogenetic and $\Hom_{\cC_\bullet}(1_{\cC_\bullet}, c)$ satisfies a certain the sheaf condition for any $c\in \cC$. 
\end{enumerate}

\end{remark}

\begin{assumption}\label{assumption:tame_geometry}
    \begin{enumerate}
            \item $\Spc_{\otimes}(\pi_0(\cC))$ is hypercomplete.
        \item The higher structure sheaf  is classical i.e., $\pi_0(\cO_\cC)\simeq \cO_\cC$. Hence $\cO_\cC\cong \cO_{\pi_0(\cC)}$.
    \end{enumerate}
\end{assumption}
\begin{lemma}
    Under Assumption~\ref{assumption:tame_geometry}, by construction, the canonical functor upgrades to $m_{\cC}\colon \cC\to D(\cO_{\pi_0(\cC)})$.
\end{lemma}
\begin{proof}
    By the hypercompleteness,  $\Sh(\Spc_{\otimes}(\pi_0(\cC)), \Mod(\bK))$ is the same as the derived category of the category of $\bK$-module sheaves over $\Spc_{\otimes}(\pi_0(\cC))$. By the second assumption, the algebra object $\cO_\cC$ is defined in the heart of $\Sh(\Spc_{\otimes}(\pi_0(\cC)), \Mod(\bK))$. This means that $\Sh(\Spc_{\otimes}(\pi_0(\cC)), \Mod(\bK))\cong D(\cO_{\pi_0(\cC)})$.
\end{proof}

\begin{comment}

The fundamental question is the following:
\begin{question}
    When is $m_\cT$ an embedding? 
\end{question}
\begin{corollary}
    If $\cT$ is a bounded derived category of coherent sheaves over a noetherian scheme, then $m_\cT$ is fully faithful.
\end{corollary}

\begin{proof}
This is essentially Balmer's reconstruction.
\end{proof}
   
\end{comment}

\subsection{Relation to Balmer's reconstruction}
Let $X$ be a Noetherian scheme. The category of perfect complexes $\Perf(X)$ over $X$ carries a standard monoidal structure and an enhancement. As observed in \cite{Balmer}, the Balmer spectrum of $\Perf(X)$ recovers $X$ as a ringed space. 
\begin{proposition}\label{prop:embedding}
    The functor $m_{\Perf(X)}$ is equivalent to the inclusion $\Perf(X)\hookrightarrow D(\cO_X)$.
\end{proposition}
\begin{proof}
Note that any Noetherian scheme satisfies Assumption~\ref{assumption:tame_geometry}.
Now the assertion follows from the fact that the sheafification of $U\mapsto \Hom(\cO_U, \cE_U)$ is $\cE$.
\end{proof}

\begin{remark}[Matsui spectrum]
    Even if one does not have a monoidal structure, there is an analogue called \emph{Matsui spectrum}, which is again a ringed space. Matsui spectrum does not reconstruct varieties in general, but has some interesting partial information. See e.g.~\cite{matsuiringedspace,HiranoOuchi1,HiranoOuchi2,ItoMatsui,Ito}. Since our construction of the canonical functor uses a special object ``monoidal unit", to construct a similar functor for Matsui spectrum, one needs to choose an object playing the role of monoidal unit. We will not explore this direction in this note, since it seems to be a bit weak for our purpose (=HMS).
\end{remark}

{\small
\bibliographystyle{alpha}
\bibliography{bibs}}

\end{document}